\documentclass[letterpaper, 10 pt, conference]{ieeeconf}  

\IEEEoverridecommandlockouts                              

\usepackage{psfrag, epsfig, amsmath,amssymb}

\usepackage{amsthm}
\usepackage{xcolor}
\usepackage[pdfpagemode=UseNone,bookmarksopen=false,colorlinks=true,urlcolor=blue,citecolor=blue,citebordercolor=blue,linkcolor=blue]{hyperref}
\newtheorem{theorem}{Theorem}
\newtheorem{proposition}[theorem]{Proposition}

\newtheorem{definition}[theorem]{Definition}

\newtheorem{lemma}[theorem]{Lemma}
\newtheorem{corollary}[theorem]{Corollary}
\newtheorem{conjecture}[theorem]{Conjecture}

\renewenvironment{proof}{\textbf{Proof:}}{\hfill$\square$}

\def\N{{\mathbb N}}
\def\EE{{\mathbb E}}
\def\E{{\mathbb E}}
\def\PP{{\mathbb P}}
\def\RR{{\mathbb R}}

\def\Poi{\text{Poi}}

\def\Ncal{{\mathcal N}}

\def\EE{{\mathbb E}}

\def\PP{{\mathbb P}}

\def\R{{\mathbb R}}

\def\N{{\mathbb N}}

\newcommand{\BEAS}{\begin{eqnarray*}}
\newcommand{\EEAS}{\end{eqnarray*}}
\newcommand{\BEA}{\begin{eqnarray}}
\newcommand{\EEA}{\end{eqnarray}}
\newcommand{\BEQ}{\begin{equation}}
\newcommand{\EEQ}{\end{equation}}
\newcommand{\BIT}{\begin{itemize}}
\newcommand{\EIT}{\end{itemize}}
\newcommand{\BNUM}{\begin{enumerate}}
\newcommand{\ENUM}{\end{enumerate}}

\usepackage{mathtools}
\DeclarePairedDelimiter\ceil{\lceil}{\rceil}

\title{\LARGE \bf
  Recovering asymmetric communities in the stochastic block model}

\author{Francesco Caltagirone, Marc Lelarge and L\'eo Miolane\thanks{The authors are with INRIA-ENS. {\tt\small f.calta@gmail.com}; {\tt\small marc.lelarge@ens.fr}; {\tt\small leo.miolane@ens.fr}}}

\begin{document}

\maketitle
\thispagestyle{empty}
\pagestyle{empty}

\begin{abstract}

We consider the sparse stochastic block model in the case where the degrees are uninformative. The case where the two communities have approximately the same size has been extensively studied and we concentrate here on the community detection problem in the case of unbalanced communities. In this setting, spectral algorithms based on the non-backtracking matrix are known to solve the community detection problem (i.e.\ do strictly better than a random guess) when the signal is sufficiently large namely above the so-called Kesten Stigum threshold. In this regime and when the average degree tends to infinity, we show that if the community of a vanishing fraction of the vertices is revealed, then a local algorithm (belief propagation) is optimal down to Kesten Stigum threshold and we quantify explicitly its performance. Below the Kesten Stigum threshold, we show that, in the large degree limit, there is a second threshold called the spinodal curve below which, the community detection problem is not solvable. The spinodal curve is equal to the Kesten Stigum threshold when the fraction of vertices in the smallest community is above $p^*=\frac{1}{2}-\frac{1}{2\sqrt{3}}$, so that the Kesten Stigum threshold is the threshold for solvability of the community detection in this case. However when the smallest community is smaller than $p^*$, the spinodal curve only provides a lower bound on the threshold for solvability. In the regime below the Kesten Stigum bound and above the spinodal curve, we also characterize the performance of best local algorithms as a function of the fraction of revealed vertices. Our proof relies on a careful analysis of the associated reconstruction problem on trees which might be of independent interest. In particular, we show that the spinodal curve corresponds to the reconstruction threshold on the tree.

\end{abstract}

\section{Introduction}
We first define the random graph model that we are going to study.

\begin{definition}[Stochastic block model (SBM)]
	Let $M$ be a $2 \times 2$ symmetric matrix whose entries are in $[0,1]$. Let $n \in \N^*$ and $p \in [0,1]$. We define the stochastic block model with parameters $(M,n,p)$ as the random graph $G$ defined by: 

	\begin{enumerate}
		\item The vertices of $G$ are the integers in $\{1, \dots, n\}$. 
		\item For each vertex $i \in \{1, \dots, n \}$ one draws independently $X_i \in \{1,2\}$ according to $\PP(X_i = 1) = p$. $X_i$ will be called the label (or the class, or the community) of the vertex $i$.
		\item For each pair of vertices $\{i,j\}$ the unoriented edge $G_{i,j}$ is then drawn conditionally on $X_i$ and $X_j$ according to a Bernoulli distribution with mean $M_{X_i,X_j})$.
	\end{enumerate}
\end{definition}

The graph $G$ is therefore generated according to the underlying partition of the vertices in two classes. Our main focus will be on the community detection problem: given the graph $G$, is it possible to retrieve the labels $X$ better than a random guess? 

We investigate this question in the asymptotic of large sparse graphs, when $n \rightarrow +\infty$ while the average degree remains fixed. Our quantitative results will then be obtained when the average degree tends to infinity.
We define the connectivity matrix $M$ as follows:

\begin{equation} \label{eq:M}
	M =
	\frac{d}{n}
	\begin{pmatrix}
		a & b \\ 
		b & c
	\end{pmatrix},
\end{equation}
where $a,b,c,d$ remain fixed as $n \rightarrow + \infty$. In this case, it is not possible to correctly classify more than a certain fraction of the vertices correctly and we will say that community detection is solvable if there is some algorithm that recovers the communities more accurately than a random guess would.

A simple argument (see below) shows that if $pa +(1-p)b \neq pb + (1-p)c$
then non-trivial information on the community of a vertex can be gained just by looking at its degree and the community detection is then solvable. In this paper, we concentrate on the case:
\BEA
\label{eq:egal}pa +(1-p)b = pb + (1-p)c = 1.
\EEA
The symmetric case where $p=1/2$ and $a=c$ has been extensively studied starting with \cite{decelle2011asymptotic} and gives rise to an interesting phenomena: if $d(1-b)^2<1$ then community detection is not solvable \cite{mossel2015symSBM1}, while if $d(1-b)^2>1$, it is solvable (in polynomial time) \cite{massoulie2014symSBM,mossel2013symSBM2}. Much less is known in the case where \eqref{eq:egal} holds and $p<1/2$.

Under condition \eqref{eq:egal}, the matrix $R$ defined by
\BEA
\label{eq:defR}R=\begin{pmatrix}
		pa & (1-p)b \\ 
		pb & (1-p)c
	\end{pmatrix}
\EEA 
is a stochastic matrix with two eigenvalues: $\lambda_1=1>\lambda_2=1-b$. If $d\lambda_2^2>1$, it is shown in \cite{neeman2014asymSBM} that it is easy to distinguish $G$ from an Erd\H{o}s-R\'enyi model with the same average degree. Although this is not formally proven in \cite{bordenave2015non}, the spectral algorithm based on the non-backtracking matrix should solve the community detection problem in this case. Here, we prove that if a vanishing fraction of labels is given, then a local algorithm (belief propagation) allows to solve the community detection problem.

The case $d\lambda_2^2<1$ (known as below the Kesten-Stigum bound) is more challenging. It is shown in \cite{neeman2014asymSBM}, that the community detection problem is still solvable for some values of $p$ but it is expected in \cite{decelle2011asymptotic} that no computationally efficient reconstruction is possible. In \cite{neeman2014asymSBM}, some bounds are given on the non-reconstructability region but they are not expected to be tight.

In this work, we are mainly interested in a regime where $d$ tends to infinity while $d\lambda_2^2=d(1-b)^2$ remains constant (in particular $a,b,c$ tend to one). Note that we first let $n$ tend to infinity and then let $d$ tend to infinity mainly in order to get explicit formulas.
In this regime, we show that for all values of $p\in (p^*,1/2)$ with $p^*=\frac{1}{2}-\frac{1}{2\sqrt{3}}$, the situation is similar to the balanced case: below the Kesten-Stigum bound, i.e.\ when $d\lambda_2^2<1$, the community detection problem is not solvable.
For $p<p^*$, we compute a function $p\mapsto\lambda_{sp}(p)<1$ such that for $d\lambda_2^2<\lambda_{sp}(p)$, the community detection problem is not solvable. As shown by \cite{neeman2014asymSBM}, there are points in the region $\lambda_{sp}(p)<d\lambda_2^2<1$ where the community detection problem is solvable but we do not expect the bound $\lambda_{sp}(p)$ to be tight, i.e.\ the information theoretic threshold for community detection should be above $\lambda_{sp}(p)$ for $p<p^*$.  

There is an important probabilistic interpretation of the matrix $R$ relating to the local structure of the SBM. As explained below, the SBM converges locally toward a labeled Poisson Galton-Watson branching process with mean offspring $d$: the label of the root is $1$ with probability $p$ and $2$ with probability $1-p$ and then conditioned on the parent's label being $i$, its children's labels are independently chosen to be $j$ with probability $R_{ij}$. A problem closely related to the detection problem in the SBM is the reconstruction problem on this random tree: given some information about the labels at depth $n$ from the root, is it possible to infer some information about the label of the root when $n\to \infty$? It is known \cite{mossel2004survey} that the Kesten-Stigum bound corresponds to census-solvability (i.e.\ knowing only the number of labels $1$ and $2$ at depth $n$ allows to get some information about the label of the root). When $d\to \infty$, we show that $\lambda_{sp}(p)$ corresponds to the solvability threshold for the reconstruction problem on the tree. We also consider the reconstruction problem where the label of each node at depth $n$ is revealed with probability $q$. Then in the region $\lambda_{sp}(p)<d\lambda_2^2<1$, we compute the minimal value of $q$ such that some information about the label of the root can be recovered from the revealed labels. Above the Kesten Stigum bound, i.e.\ when $1<d\lambda_2^2$, this minimal value is 0.

In Section~\ref{section:community_detection}, we define the community detection problem and its variation when some labels are revealed. We also give several equivalent formulation for its solvability. In Section~\ref{sec:above}, we give our main results about reconstruction above the Kesten-Stigum bound and in Section~\ref{sec:nonrec}, we describe what happens below the Kesten-Stigum bound. Section~\ref{sec:pbrectree} defines the various notions of solvability for the problem of reconstruction on trees and gives our main result for this problem. The technical proofs are given in the subsequent sections. We first use the cavity method on trees in Section~\ref{section:cavity_method} and then relate these results to the original problem of community detection in Section~\ref{sec:proofsbm}.

\section{Community detection in the stochastic block model} \label{section:community_detection}

We are interested into inferring the labels $X$ from the graph $G$.  To do so, we aim at constructing an estimator (i.e.~a function of the observation $G$) $T(G)=(T_1(G) , \dots, T_n(G))\in \{1,2\}^n$, such that $T_i(G)$ is `close' to $X_i$. We will measure the performance of $T$ using the `rescaled average success probability' defined as follows
\begin{align}
	P_{succ}(T) = \frac{1}{n} \sum_{i=1}^n &\Big( \PP(T_i(G) = 1 | X_i = 1) \nonumber \\
				&+ \PP(T_i(G) = 2 | X_i = 2) - 1 \Big)\label{eq:defpsucc}
\end{align}

The `$-1$' is here to re-scale the success probability and ensures that $P_{succ} = 0$ for `dummy estimators' (i.e.\ estimators that not depends on the observed graph $G$).

The optimal test with respect to this metric is
$$
T^{opt}_i(G) =
\begin{cases}
	1 & \text{if } \log(\frac{\PP(X_i = 1 \vert G)}{\PP(X_i = 2 \vert G)}) \geq \log(\frac{p}{1-p}) \\
	2 & \text{otherwise} \\
\end{cases}
$$

Let $s_0$ be uniformly chosen among the vertices of $G$, independently of all other random variables. The maximal achievable rescaled success probability reduces then to

\BEAS
P_{succ}(T^{opt}) &=& \PP(T^{opt}_{s_0}(G) = 1 | X_{s_0} = 1)\\
&&+ \PP(T^{opt}_{s_0}(G) = 2 | X_{s_0} = 2) - 1
\EEAS

We define a notion of solvability for the community detection problem.

\begin{definition}
	We say that the community detection problem on a given stochastic block model is solvable if
	$$
	\liminf_{n \to \infty} P_{succ}(T^{opt}) > 0.
	$$
\end{definition}

Another popular measure of performance is the overlap as defined in \cite{neeman2014asymSBM}:
\begin{align*}
	&{\rm overlap}(T(G),X) \\
&= \EE\left[ (T_{s_0}-1)(X_{s_0}-1)\right]
-\EE\left[ T_{s_0}-1\right]\EE\left[X_{s_0}-1\right] \\
&+\EE\left[ (2-T_{s_0})(2-X_{s_0})\right]
-\EE\left[2-T_{s_0}\right]\EE\left[2-X_{s_0}\right].
\end{align*}

A simple computation shows that
\BEAS
{\rm overlap}(T^{opt},X) &=& p(1-p)P_{succ}(T^{opt}).
\EEAS

We have another equivalent characterization for solvability given by the following proposition:
\begin{proposition} \label{prop:equivalence_dtv_psucc}
	We have
	$$
	P_{succ}(T^{opt}) = D_{TV}(P_{1},P_{2}),
	$$
	where $P_1$ and $P_2$ denote the conditional distribution of the graph $G$, conditionally respectively on $X_{s_0} = 1$ and $X_{s_0} = 2$, where $s_0$ is a uniformly chosen random vertex of $G$.
\end{proposition}
\begin{proof}
  The set of estimators of $X_{s_0}$ is precisely $\{ 1 + 1_A | A \text{ measurable set} \}$. Consequently
  \BEAS
  P_{succ}(T^{opt}) &=& \sup_{A \text{ measurable}} \PP(G \in A^c | X_{s_0} = 1) \\
        &&+ \PP(G \in A | X_{s_0} = 2) - 1 \\
  &=& \sup_{A \text{ measurable}} \PP_1(A^c) + \PP_2(A) - 1\\
  &=& \sup_{A \text{ measurable}} \PP_2(A) - \PP_1(A) \\
		&=& D_{TV}(P_1,P_2) 
\EEAS
\end{proof}

In our setting \eqref{eq:M}, the asymptotic degree distribution of a given vertex is a Poisson random variable. Let $i \in \{1, \dots, n\}$ be a vertex of $G$, we will denote by $d_i$ its degree. Then we have
\begin{align*}
	d_i \xrightarrow[n \to +\infty]{(d)} \Poi(d(pa+(1-p)b)) \text{ cond. on } \{X_i = 1 \}
	\\
	d_i \xrightarrow[n \to +\infty]{(d)} \Poi(d(pb+(1-p)c)) \text{ cond. on } \{X_i = 2 \}
\end{align*}

If the asymptotic average degrees differ from class 1 to class 2, we see easily that the problem is solvable.

\begin{lemma}
	If $pa + (1-p)b \neq pb + (1-p)c$ then the community detection problem is solvable.
\end{lemma}

\begin{proof}
  Using the definition of solvability in term of the total variation distance, we have:
  \BEAS
  \lefteqn{\liminf_{n \to \infty} D_{TV}(P_1,P_2) }\\
  &=& \liminf_{n \to \infty} \sup_{\text{event } A} | P_1(A) - P_2(A) |
		\\
		& \geq& \liminf_{n \to \infty} \sup_{B \subset \N} | P_1(\{d_{s_0} \in B \}) - P_1(\{d_{s_0} \in B \}) |
		\\
		&=& D_{TV} (\Poi(d(pa + (1-p)b)), \Poi(d(pb + (1-p)c))) \\
                &>& 0
\EEAS
\end{proof}

In the rest of the paper, we will always assume that \eqref{eq:egal} is valid,
so that the average degree in the graph is $d$.

We also define
\BEA
\label{eq:deflambda}\lambda = d(1-b)^2=d\lambda_2^2,
\EEA
where $\lambda_2$ is the second largest eigenvalue of the matrix $R$ defined in \eqref{eq:defR}. As explained in the introduction, the Kesten-Stigum threshold corresponds to $\lambda =1$. The larger $\lambda$ is, the stronger the signal is, so that $\lambda$ can be interpreted as a ``signal to noise ratio''.  Most of our results below will be obtained in the limit where first $n$ tends to infinity and then $d$ tends to infinity but the parameter $\lambda$ will always remain fixed, as well as the parameter $p\in [0,1/2]$ corresponding to the proportion of nodes in community one.

We now introduce a variation of the standard community detection problem where a fraction $q$ of the vertices have their labels revealed. More formally, in this setting the label $X_v$ of each vertex $v\in \{1,\dots,n\}$ is observed with probability $q\in [0,1]$ independently of everything else and the estimator $T(G,q)=(T_1(G,q),\dots, T_n(G,q))$ is then a function of the observed graph $G$ and the observed labels. The probability of success is again defined by \eqref{eq:defpsucc} where the $T_i(G)$'s are replaced by the $T_i(G,q)$'s and the optimal test $T^{opt}(G,q)$ is then
$$
T^{opt}_i(G,q) =
\begin{cases}
	1 & \text{if } \log(\frac{\PP(X_i = 1 \vert G,Y)}{\PP(X_i = 2 \vert G,Y)}) \geq \log(\frac{p}{1-p}) \\
	2 & \text{otherwise,} \\
\end{cases}
$$
where $Y$ denotes the (random) set of observed vertices.
Note that we have
$$
\liminf_{n \to \infty} P_{succ}(T^{opt}(G,q))\geq q \geq 0,
$$
so that the notion of solvability is unclear as soon as $q>0$.

We end this section by some technical definitions.
In order to state our main results, we need to define the following function $G$ from $\RR_+$ to $\RR$:
\BEAS
G(\mu) = \frac{\lambda}{(1-p)^2} \EE\Big[ \frac{1}{p + (1-p)\exp(\sqrt{\mu}Z - \mu/2)} - 1 \Big],
\EEAS
where $Z \sim \mathcal{N}(0,1)$. Note that $G$ is also a function of the parameters $\lambda$ and $p$ which are considered as fixed.

\begin{definition}[Spinodal curve]\label{def:spinodal}
	The spinodal curve is defined as the function
\BEA
\label{eq:def_spinodal} \lambda_{sp}: p \mapsto \sup \big\{\lambda \geq 0 \ | \ 0 \text{ is the unique fixed point of } G \big\}.
\EEA
\end{definition}
Let us define
\BEA
\label{eq:defp*}p^*=\frac{1}{2} - \frac{1}{2\sqrt{3}}.
\EEA

The following conjecture shows that $\lambda_{sp}$ is well defined and summarize its main properties.
\begin{conjecture} \label{lem:fixed_points}
  \BIT
\item[(i)] If $\lambda>1$, then $G$ has two fixed points: $0$ and $\alpha > 0$. Moreover, $0$ is unstable and $\alpha$ is stable.
\item[(ii)] For $p^* \leq p \leq 1/2$, we have $\lambda_{sp}(p) = 1$.
  \item[(iii)] For $0\leq p<p^*$, we have $\lambda_{sp}(p)<1$ and if $\lambda_{sp}(p) < \lambda < 1$, then $G$ has three fixed points: $0 < \beta < \alpha$. Moreover, $0$ and $\alpha$ are stable and $\beta$ is unstable.
\EIT
\end{conjecture}
\begin{proof}
  The analysis of the function $G$ seems challenging and we were only able to verify Conjecture~\ref{lem:fixed_points} numerically. The exact value of $p^*$ follows from the following argument: a small $\mu$ expansion of the function $G$ gives
  \BEAS
G(\mu)\approx \lambda \mu +\frac{\lambda}{2}(1-6p(1-p))\mu^2.
\EEAS
In particular, for $\lambda=1$, in order for $G$ to have three fixed points, we need to have
$1-6p(1-p)<0$, i.e. $p<p^*$.
\end{proof}
\begin{figure}[thpb]
  \centering
  \includegraphics[width=8cm]{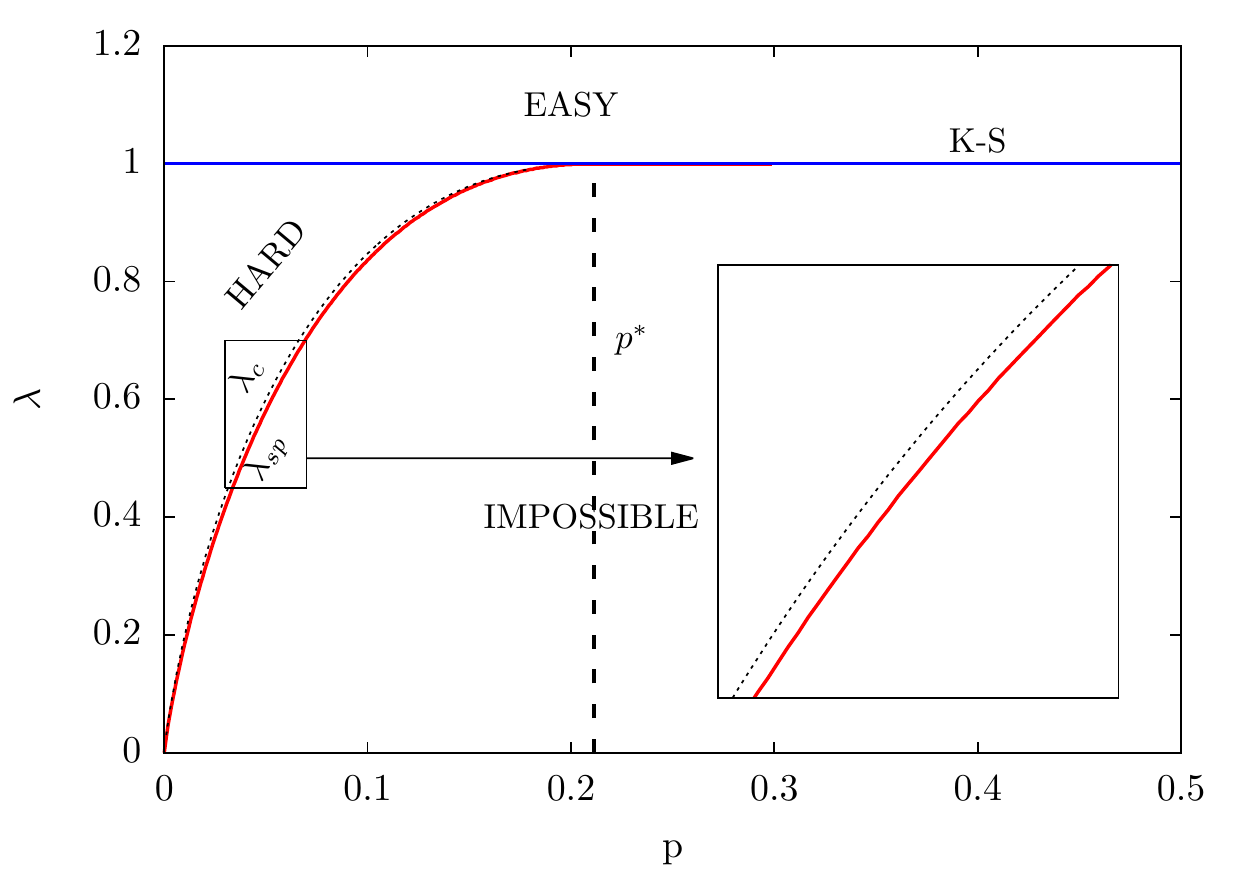}
  \caption{Phase diagram for the asymmetric community detection problem. The easy phase follows from \cite{bordenave2015non}, the impossible phase below the spinodal curve (red curve) is proved in this paper and the hard phase is a conjecture. The dotted curve corresponding to $\lambda_c(p)$ is the conjectured curve for solvability of the community detection problem (see discussion in Section~\ref{sec:nonrec}).}
  \label{fig:phase_diagram}
\end{figure}

We summarize on the following phase diagram (Figure~\ref{fig:phase_diagram}) our main results:
\begin{itemize}
	\item Above the Kesten-Stigum bound (blue line), reconstruction is possible by a local algorithm given that an arbitrary small fraction of the labels is revealed. Moreover, the local algorithms (with this arbitrary small side information) achieve then the best possible performance without side information (see Proposition~\ref{prop:ks}).
	\item Between the blue and the red line, we show that local algorithms are efficient for reconstruction when a certain fraction of labels is revealed (see Proposition~\ref{prop:belowks}).
	\item We show that reconstruction is impossible below the spinodal curve (red line), see Proposition~\ref{prop:below_sp}.
\end{itemize}

\section{Reconstructability above the Kesten-Stigum bound}\label{sec:above}

We first consider the case $\lambda>1$. 

\begin{proposition}\label{prop:ks}
	If $\lambda > 1$, then we have
        \BEA
 \label{eq:b}       \limsup_{d\to\infty}\limsup_{n\to \infty}{P}_{succ}(T^{opt}) \leq 2\PP(\Ncal(\alpha/2,\alpha)>0)-1 
        \EEA
        where $\alpha$ is the stable fixed point in Conjecture~\ref{lem:fixed_points} (i).

        Moreover, for all $0<q<1$, we have
        \BEAS
		\liminf_{d\to\infty}\liminf_{n\to \infty}{P}_{succ}(T^{opt}(G,q)) \geq 2\PP(\Ncal(\alpha/2,\alpha)>0)-1.
        \EEAS
\end{proposition}

\begin{figure}[thpb]
  \centering
  \includegraphics[width=7cm]{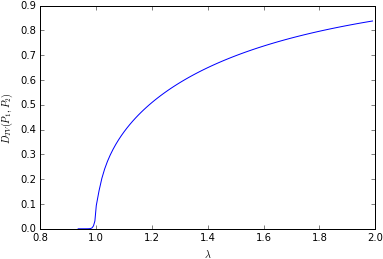}
  \caption{Lower bound for the probability to recover the true label of a typical vertex by an optimal local algorithm with side information for $p>p^*$ as a function of $\lambda$ (i.e.\ function $2\PP(\Ncal(\alpha/2,\alpha)>0)-1$ for $p=0.25>p^*$).}
  \label{fig:1}
\end{figure}

In words, we see that if a vanishing fraction of the labels is revealed, then the probability to recover the true label of a typical vertex by the optimal algorithm is $2\PP(\Ncal(\alpha/2,\alpha)>0)-1$. Indeed, we believe that \eqref{eq:b} should be an equality.
On Figure~\ref{fig:1}, we give a drawing of this curve as a function of $\lambda$ for $p=0.25>p^*$ and on Figure~\ref{fig:2} for $p=0.005<p^*$. Note that at this stage, we only gave an interpretation of the curve for $\lambda>1$. We deal with the case $\lambda<1$ in the next section.


Before that, we give a result which shows that if a vanishing fraction of the labels is revealed then the optimal recovery is achieved by a local algorithm. Similar results in the case where \eqref{eq:egal} does not hold have been proved in \cite{mossel2015density}. 
In the large degree $d$ regime, our result improves Proposition 3 in \cite{kanade2014global} which deals only with the case $p=0.5$ and $\lambda$ larger than a large constant $C$. The fact that local algorithms are very efficient as soon as $q>0$ (even optimal in the limit $q\to 0$) leads to linear time algorithms for community detection (when some labels are revealed). Indeed from a practical perspective, we believe that our analysis carries over to the labeled stochastic block model \cite{heimlicher2012community,lelarge2015reconstruction}. It is then possible to devise new clustering algorithms based on a similarity graph which are shown to be optimal for a wide range of models \cite{saade2016clustering} and also local semi-supervised learning clustering algorithms, see \cite{saade2016fast} for more details in this direction. 

We now define local algorithms. For an integer $t$, a test $T(G,q)=(T_1(G,q),\dots, T_n(G,q))$ is $t$-local if each $T_i(G,q)$ is a function of the graph $B_t(G,i)$ induced by the vertices of $G$ whose distance from $i$ is at most $t$. We denote by $Loc_t$ the set of $t$-local tests and by $Loc=\cup_{t\geq 0}Loc_t$ the set of local tests.

\begin{proposition} \label{prop:lower_bound_ks}
  If $\lambda >1$, then we have for all $0 < q \leq 1$,
  \BEAS
\sup_{T\in Loc}\lim_{d\to\infty}\lim_{n\to \infty}{P}_{succ}(T(G,q)) \geq 2\PP(\Ncal(\alpha/2,\alpha)>0)-1,
  \EEAS
  where $\alpha$ is the stable fixed point in Conjecture~\ref{lem:fixed_points} (i).
\end{proposition}

Note in particular that as a vanishing fraction of labels is revealed, i.e. $q\to 0$, the best local algorithm performs at least as well as the optimal algorithm with no revealed labels. An explicit description of an optimal local test is given in the proof of this proposition.



\section{Non-reconstructability below the spinodal curve}\label{sec:nonrec}

We now state our second main result:

\begin{proposition} \label{prop:below_sp}
	If $\lambda < \lambda_{sp}(p)$ then 
	$$
	\lim_{d \to \infty} \lim_{n \to \infty} P_{succ}(T^{opt}) = 0
	$$
\end{proposition}

If $\lambda_c(p)=\inf\{\lambda,\: \mbox{community detection is solvable}\}$ denotes the solvability threshold, then Proposition~\ref{prop:below_sp} implies that $\lambda_{sp}(p)\leq \lambda_c(p)$. Moreover thanks to \cite{bordenave2015non}, the Kesten Stigum threshold is an upper bound on the solvability threshold so that we have $\lambda_c(p)\leq 1$.
For $p\geq p^*$ defined in \eqref{eq:defp*}, the spinodal curve is equal to the Kesten-Stigum threshold by Conjecture~\ref{lem:fixed_points} (ii), so that we have in this case $\lambda_c(p)=1$ and moreover as soon as the community detection problem is solvable, it is solvable in polynomial time thanks to the results in \cite{bordenave2015non}. Figure~\ref{fig:1} is valid for $\lambda<1$. 
However for $p<p^*$, there is a gap between the spinodal curve and Kesten-Stigum threshold and we conjecture that $\lambda_{sp}(p)<\lambda_c(p)<1$, see Figure~\ref{fig:phase_diagram}. 
In the case of dense graphs (where the average degree $d$ is of order $n$), the value of $\lambda_c(p)$ has been computed in the recent works \cite{krzakala2016lowrank} and \cite{barbier2016mutual}.
We conjecture that their expression (used in Figure~\ref{fig:phase_diagram}) is still valid for sparse graphs in the large degree regime.

\begin{figure}[thpb]
  \centering
  \includegraphics[width=7cm]{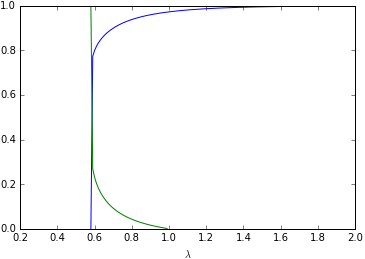}
  \caption{Necessary fraction of revealed labels (green) and corresponding lower bound of probability to recover the true label of a typical vertex by an optimal local algorithm (blue) for $p<p^*$, i.e.\ functions $\frac{\beta p(1-p)}{\lambda}$ (in green) and $2\PP(\Ncal(\alpha/2,\alpha)>0)-1$ (in blue) appearing in Proposition~\ref{prop:belowks} as a function of $\lambda$ for $p=0.05<p^*$.}
  \label{fig:2}
\end{figure}

\begin{proposition}\label{prop:belowks}
  Consider the case where $p<p^*$ and $\lambda_{sp}(p)<\lambda<1$.
  As soon as $q>\frac{\beta p(1-p)}{\lambda}$, we have 
\BEAS
\lefteqn{\lim_{d \to \infty}\liminf_{n \to \infty} P_{succ}(T^{opt}(G,q))}\\
&\geq& \sup_{T\in Loc}\lim_{d \to \infty}\liminf_{n \to \infty} P_{succ}(T(G,q))\\
&\geq& 2\PP(\Ncal(\alpha/2,\alpha))-1,
\EEAS
where $\alpha$ and $\beta$ are the fixed points defined in Conjecture~\ref{lem:fixed_points} (iii).
\end{proposition}

In the regime of Proposition~\ref{prop:belowks} ($\lambda_{sp}(p)<\lambda<1$ and $q>\frac{\beta p(1-p)}{\lambda}$), we believe that local algorithms are indeed optimal. Figure~\ref{fig:2} illustrates the case $p<p^*$ with $p=0.05$ for which we have $\lambda_{sp}(0.05)\approx 0.58$. 
Also, if the number of revealed entries is sufficiently high (i.e.\ above the green curve) then local algorithms provide a great improvement in the probability of successfully recovering the label of a typical vertex (the blue curve). An explicit description of a local algorithm achieving the lower bound in Proposition~\ref{prop:belowks} is provided in the proof. 

\section{Reconstruction on trees}\label{sec:pbrectree}

In this section we are going to state the analogous of the well known local convergence of the Erd\H{o}s-R\'enyi random graph towards the Galton-Watson branching process, in term of labeled graphs. The labeled stochastic block model $(G,X)$ will converges locally towards a random labeled tree. We have to introduce first the notion of pointed labeled graphs.

\begin{definition}[Pointed labeled graphs]
	\begin{itemize}
		\item A pointed labeled graph is a triple $G=(g,s_0,x)$ where $g$ is a countable, locally finite and connected graph, $s_0$ is a distinguished vertex of $g$ called the root of the graph and $x=(x_s)_{s \in V_g} \in \{1,2\}^{V_g}$ are the labels of the vertices.

		\item Two pointed labeled graphs are equivalent if there exists a graph isomorphism between them, that preserves the root and the labels.

		\item We define, for $r \in \N$, $[G]_r$, the ball of radius $r$ of $G$, as the pointed labeled graph induced by the root of $G$ and all the vertices at distance at most $r$ from the root.
	\end{itemize}
\end{definition}

The randomly rooted stochastic block model $(G,s_0,X)$ is therefore a random pointed labeled graph, that we will denote $SBM_n$ from now. We will also be interested in a second family of random pointed labeled graphs, that will corresponds to the local limits of stochastic block models.

\begin{definition}[Labeled Poisson Galton-Watson branching process]
	Let $A = 
	\begin{pmatrix}
		\delta & 1-\delta \\
		1-\delta' & \delta' \\
	\end{pmatrix}$ 
	(where $\delta, \delta' \in [0,1]$) be a transition matrix.
	The labeled Poisson Galton-Watson branching process with parameters $(A,p,d)$ is a random pointed labeled graphs $(T,s_0,X)$, where
	\begin{itemize}
		\item $(T,s_0)$ is a Galton-Watson tree with offspring distribution $\Poi(d)$ rooted at $s_0$.
		\item The labels $X$ of the vertices of $T$ are then chosen as follows:
			\begin{enumerate}
				\item
					The label of the root $X_{s_0} \in \{1,2\}$ is chosen accordingly to $\PP(X_{s_0} = 1) = p$.
				\item
					Given the label $X_{p}$ of the parent $p$ of a node $s$, the
					probability that $X_s = i \in \{1,2\}$ is equal to $A_{X_p, i}$
					independently from all other random variables.
			\end{enumerate}
	\end{itemize}
\end{definition}

In the following, we will denote $GW = (T,s_0,X)$, the labeled Galton-Watson branching process with parameters $(R,d,p)$ with $R$ defined by \eqref{eq:defR}. The next result state that $SBM_n$ converges locally toward GW.

\begin{theorem} \label{th:sbm_to_gw}
	Let $F$ be a (positive or bounded) function of pointed labeled graphs, that depends only on the ball of radius $r$. Then
	$$
	\E \big[ F(SBM_n) \big] \xrightarrow[n \to \infty]{} \E \big[ F(GW) \big]
	$$
\end{theorem}

Theorem~\ref{th:sbm_to_gw} leads us to study the reconstruction problem on random trees, which will be connected later to the community detection problem. 
A main ingredient for our proof will be the analysis of this well-studied problem of reconstruction on trees \cite{evans2000broadcasting,mossel2003information,mossel2004survey}. In the rest of this section, we define the reconstruction problem on trees and give the required results.

We consider here $GW=(T,s_0,X)$ the labeled Poisson Galton-Watson branching process with parameters $(R,p,d)$. We denote $L_n = \{ v \in V_T \ | \ d(s_0,v) =n \}$, the set of vertices at distance $n$ from the root. We define then $X^{(n)}= (X_s)_{s \in L_n}$ and $c^{(n)} = (c_1^{(n)}, c_2^{(n)}) = (\# \{ s \in L_n \ | \ X_s =i \})_{i=1,2}$. 
We also define a random subset $E_n$ of the node at depth $n$ as follows: let $q \in [0,1]$ and for $n \in \N$, let $E_n$ be the random subset of $L_n$ obtained by including in $E_n$ each vertex $s\in L_n$ independently with probability $q$. 

We have three kinds of reconstruction problems.

\begin{definition}[Solvability, q-solvability and census solvability]
	We say that the reconstruction problem is solvable if $$\liminf_{n \to \infty} D_{TV}(P^{(n)}_1, P^{(n)}_2) > 0,$$ where $P^{(n)}_i$ denotes the conditional distribution of $X^{(n)}$ given $X_{s_0}=i$. One defines analogously q-solvability (respectively census solvability) by replacing $P^{(n)}_i$ by $P^{(n,q)}_i$ (respectively $\tilde{P}^{(n)}_i$), the conditional distribution of $(E_n,(X_s)_{s \in E_n})$ (respectively $c^{(n)}$) given $X_{s_0}=i$.
\end{definition}

Solvability corresponds thus to the special case $q=1$. Obviously, census solvability and q-solvability imply solvability, but we will see that solvability does not always imply census solvability.

Similarly to the stochastic block model case, this characterization of solvability in term of total variation can be rewritten in term of the maximal achievable success probability for the estimation of $X_{s_0}$, given $X^{(n)}$ (or $c^{(n)}$). We define the rescaled success probability of an estimator $F$ as
\begin{align*}
	P_{succ}(F) &= \PP(F(X^{(n)}) = 1 | X_{s_0} = 1) \\
				&+ \PP(F(X^{(n)}) = 2  | X_{s_0} = 2) - 1.
\end{align*}

The maximal rescaled success probability is then defined as $\Delta_n = \sup_{F} P_{succ}(F)$ where the supremum is taken over all measurable function of $X^{(n)}$. Even though we defined these quantities for the solvability problem, these definitions and the following result can be straightforwardly extended to q-solvability and census-solvability. The following lemma is the analog of Proposition~\ref{prop:equivalence_dtv_psucc}.

\begin{lemma} \label{lem:dtv_psucc}
	$$
	\Delta_n = D_{TV}(P^{(n)}_1,P^{(n)}_2) 
	$$
\end{lemma}

We recall here the census-solvability criterion for our particular case (which is a straightforward extension of the results presented in \cite{mossel2004survey}).

\begin{theorem} \label{th:KS1}
  We consider the Poisson Galton-Watson branching process with parameter $(R,p,d)$.
  If $\lambda > 1$, then the problem is census-solvable and q-solvable for all $0 < q \leq 1$. If $\lambda < 1$, then the problem is not census-solvable.
\end{theorem}

In the large $d$ limit, we are able to get more quantitative results. We define
\BEAS
P_{opt} &=& \lim_{d\to\infty}\lim_{n\to \infty}D_{TV}(P_1^{(n)},P_2^{(n)}),\\
P^{(q)}_{opt} &=&\lim_{d \to \infty}\lim_{n \to \infty}  D_{TV}(P_1^{(n,q)},P_2^{(n,q)}).
\EEAS
\begin{proposition} \label{prop:result_tree}
We consider the Poisson Galton-Watson branching process with parameter $(R,p,d)$.
\begin{itemize}
	\item If $\lambda > \lambda_{sp}(p)$, then we have $P_{opt}= 2\PP(\Ncal(\alpha/2,\alpha) >0)-1 >0$,
		where $\alpha$ is the stable fixed point in Conjecture~\ref{lem:fixed_points}.

	\item If $\lambda < \lambda_{sp}(p)$ then $P_{opt} = 0$.

	\item If $\lambda_{sp}<\lambda<1$, then we have for $q>\frac{\beta p(1-p)}{\lambda}$,
		$P^{(q)}_{opt}= 2\PP(\Ncal(\alpha/2,\alpha)>0)-1>0$, while for $q<\frac{\beta p(1-p)}{\lambda}$, we have	$P^{(q)}_{opt}=0$,
		where $\alpha$ and $\beta$ are the fixed points defined in Conjecture~\ref{lem:fixed_points} (iii).
\end{itemize}
\end{proposition}
In particular, Proposition~\ref{prop:result_tree} shows that the spinodal curve is the solvability threshold for the reconstruction on trees. Proposition~\ref{prop:result_tree} is proved at the end of the next section.

\section{Cavity method on trees} \label{section:cavity_method}

Our approach here is closely related to the one of \cite{montanari2015finding} which studies the problem of finding one single community. We establish rigorously the ``cavity equations'', a recursive method to compute marginals, that originate from statistical physics.

We consider here the labeled branching process $GW=(T,s_0,X)$ with parameter $(R,d,p)$.
In order to obtain quantitative results, we will be interested in the asymptotic of large degrees $d \to \infty$ while $\lambda$ remains fixed. We also define $\epsilon = 1-b = \sqrt{\frac{\lambda}{d}}$. We then have,

$$
\begin{cases}
	a = 1 + \frac{1-p}{p}\epsilon \\
	b = 1 - \epsilon \\
	c = 1 + \frac{p}{1-p}\epsilon
\end{cases}
$$

\subsection{The cavity recursions}

Let $r \in \N^*$. For a given vertex $s$ of $GW=(T,s_0,X)$, we note $\mathcal{T}_s$ the subtree induced by $s$ and its progeny. 

With a slight abuse of notation, we write $p(x)=p^{1_{x= 1}} (1-p)^{1_{x= 2}}$.  We define also $\psi(x,y)=a^{1_{y=x=1}} b^{1_{y \neq x}} c^{1_{y=x=2}}$.
We now define recursively the ``message'' of the vertex $s \in [T]_r$ as the following function from $\{1,2\}\to \RR$:
\begin{align*}
	&\nu_r^s: x_s \in \{1,2\} \mapsto
	\\
	&\begin{cases}
	p(x_s) 1_{x_s = X_s} & \text{ if } s \in E_r \\
	p(x_s) & \text{ if } s \in L_r \setminus E_r \\
	p(x_{s}) \prod_{s \to v} \sum_{x_v}  \psi(x_{s},x_v)  \nu_r^v(x_v)& \text{ if } s \notin L_r
\end{cases}
\end{align*}

\begin{lemma} \label{lem:explicit_message}
	For all $s \in [T]_r \setminus L_r$,
\begin{align*}
	\nu_r^s(x_s)
	\propto p(x_{s}) \sum_{\overset{(x_v)_{v \in \mathcal{T}_s \cap [T]_r \setminus \{s\}}}{(x_v)_{v \in E_r} = (X_v)}} \prod_{(i \to j) \in [T]_r \cap \mathcal{T}_s} R_{x_i,x_j} 
\end{align*}
where $\propto$ means equality up to a multiplicative constant that is independent of $x_s$.
\end{lemma}

\begin{proof}
	We show this lemma by induction on $d(s_0,s)$.

	For a vertice $s$ at distance $d(s_0,s)=r-1$ from the root, the lemma follows from the definition of $\nu_r^s$. Suppose now that the lemma is true for all vertices at distance $r'<r$ from the root $s_0$. Let $s$ be a vertex at distance $r'-1$ from $s_0$. Then, by induction,
	\begin{align*}
		&\nu_r^s(x_s) = p(x_{s}) \sum_{(x_v)_{s \to v}} \prod_{s \to v} \psi(x_{s},x_v)  \nu_r^v(x_v) \\
		&= p(x_{s})\!\!\! \sum_{(x_v)_{s \to v}} \prod_{s \to v} \psi(x_{s},x_v)  
		p(x_{v})
	\!\!\! \!\!\!\!\! \sum_{\overset{(x_u)_{u \in \mathcal{T}_v \cap [T]_r \setminus \{v\}}}{(x_u)_{u \in E_r} = (X_u)}} \prod_{(i \to j) \in [T]_r \cap \mathcal{T}_v} \!\!\! \!\!\!\!\!\!  R_{x_i,x_j} \\
	&= p(x_{s}) \!\!\! \sum_{(x_v)_{s \to v}} \prod_{s \to v} R_{x_{s},x_v}
	\!\!\!\!\!\!  \sum_{\overset{(x_u)_{u \in \mathcal{T}_v \cap [T]_r \setminus \{v\}}}{(x_u)_{u \in E_r} = (X_u)}} \prod_{(i \to j) \in [T]_r \cap \mathcal{T}_v} \!\!\!\!\!\!  R_{x_i,x_j} \\
	&=p(x_{s}) \sum_{\overset{(x_v)_{v \in \mathcal{T}_s \cap [T]_r \setminus \{s\}}}{(x_v)_{v \in E_r} = (X_v)}} \prod_{(i \to j) \in [T]_r \cap \mathcal{T}_s} R_{x_i,x_j} 
	\end{align*}
\end{proof}

\begin{lemma} \label{lem:bp}
	$$
	\nu_r^{s_0}(x_{s_0}) = \PP(X_{s_0} = x_{s_0} | T, E_r,(X_v)_{v \in E_r})
	$$
\end{lemma}

\begin{proof}
	The structure of $T$ outside of $[T]_r$ does not provide any information about the labels of the vertices, thus, using lemma~\ref{lem:explicit_message}
\begin{align*}
	&\PP \big(X_{s_0}=x_{s_0} \ | \ T, E_r, (X_s)_{s \in E_r} \big) 
	\\
	&=p(x_{s_0})
	\sum_{\overset{(x_s)_{s \in [T]_r \setminus \{s_0\}}}{(x_s)_{s \in E_r} = (X_s)}} \prod_{(i \to j) \in [T]_r} R_{x_i,x_j} 
	\\
	&=\nu_r^{s_0}(x_{s_0})
\end{align*}

\end{proof}

If we write $\xi^s_r = \log(\frac{\nu_r^s(1)}{\nu_r^s(2)})$, the recursive definition of the messages gives

\begin{equation} \label{eq:xi_recursion}
	\xi^{s_0}_{r} = h + \sum_{s, s_0 \to s} f(\xi_{r}^{s}),
\end{equation}
where $h= \log(\frac{p}{1-p})$ and $f: x \mapsto \log \frac{a e^x + b}{b e^x +c}$.


\begin{definition} \label{def:xi_tree}
	We define $P_r$ as the law of $\xi^{s_0}_{r} =  \log \big( \frac{\PP(X_{s_0} = 1 | T, E_r, (X_s)_{s \in E_r})}{\PP(X_{s_0} = 2 | T, E_r, (X_s)_{s \in E_r})} \big)$. We denote $P_r^{(1)}$ and $P_r^{(2)}$ the laws of $\xi^{s_0}_r$ respectively conditionally on $\{X_{s_0} = 1 \}$ and $\{X_{s_0}=2 \}$.
\end{definition}

Lemma~\ref{lem:explicit_message} shows that, conditionally on $\{s_1, \dots, s_L\}$ being the children of $s_0$, $\xi_{r}^{s_1}, \dots, \xi_{r}^{s_L}$ are independent, identically distributed according to $P_{r-1}$. 


Then, \eqref{eq:xi_recursion} leads to the following distributional recursion:
\begin{proposition}
	$$
	\xi_r \overset{(d)}{=} h + \sum_{i=1}^L f(\xi_{r-1,i})
	$$
	where $\xi_r \sim P_r$, $L \sim \Poi(d)$, $\xi_{r-1,k} \sim P_{r-1}$ are independent random variables, 
\end{proposition}

 Using similar arguments as before, \eqref{eq:xi_recursion} implies also

\begin{proposition}[Cavity equations]

	\begin{equation} \label{eq:rec1}
		\xi_r^{(1)} \overset{(d)}{=} h + \sum_{i=1}^{L_{1,1}} f(\xi_{r-1,i}^{(1)}) + \sum_{i=1}^{L_{1,2}} f(\xi_{r-1,i}^{(2)})
	\end{equation}

	where $\xi_r \sim P_r^{(1)}$, $L_{1,1} \sim \Poi(pad)$, $L_{1,2} \sim \Poi((1-p)bd)$, $\xi_{r-1,i}^{(1)} \sim P_{r-1}^{(1)}$, $\xi_{r-1,i}^{(2)} \sim P_{r-1}^{(2)}$, and all these variables are independent.

	\begin{equation} \label{eq:rec2}
		\xi_r^{(2)} \overset{(d)}{=} h + \sum_{i=1}^{L_{2,1}} f(\xi_{r-1,i}^{(1)}) + \sum_{i=1}^{L_{2,2}} f(\xi_{r-1,i}^{(2)})
	\end{equation}
	where $\xi_r \sim P_r^{(2)}$, $L_{2,1} \sim \Poi(pbd)$, $L_{2,2} \sim \Poi((1-p)cd)$, $\xi_{r-1,i}^{(1)} \sim P_{r-1}^{(1)}$, $\xi_{r-1,i}^{(2)} \sim P_{r-1}^{(2)}$, and all these variables are independent.

\end{proposition}

\subsection{Gaussian limit}

We study the recursions \eqref{eq:rec1} and \eqref{eq:rec2}. More precisely, we determine the limits of $P_r^{(1)}$ and $P_r^{(2)}$ when $d \to \infty$ and $\lambda$ remains fixed (see the beginning of section~\ref{section:cavity_method}). 

First of all, we are going to show that $\xi^{(1)}_1$ and $\xi^{(2)}_1$ converge towards Gaussian distributions.

\begin{lemma} \label{lem:init}
	\begin{align*}
		\xi_1^{(1)} \xrightarrow[d \to + \infty]{W} \mathcal{N}(h + \frac{\mu_1}{2}, \mu_1) \\
		\xi_1^{(2)} \xrightarrow[d \to + \infty]{W} \mathcal{N}(h - \frac{\mu_1}{2}, \mu_1) 
	\end{align*}
	where $\xrightarrow[]{W}$ denote the convergence in the sense of the Wasserstein metric and $\mu_1 = \frac{q \lambda}{p(1-p)}$.
\end{lemma}

\begin{proof}
	We will only prove the convergence for $\xi^{(1)}_1$, the convergence for $\xi^{(2)}_1$ can be obtained analogously.
	We have
	$$
	\xi^{(1)}_0 =
	\begin{cases}
		+ \infty & \text{ with probability } q \\
		h & \text{ with probability } 1-q \\
	\end{cases}
	$$
	$$
	\xi^{(2)}_0 =
	\begin{cases}
		- \infty & \text{ with probability } q \\
		h & \text{ with probability } 1-q \\
	\end{cases}
	$$

	Therefore, the recursion \eqref{eq:rec1} gives
	$$
	\xi^{(1)}_1 \overset{(d)}{=} h + L \log(\frac{a}{b}) + L' \log(\frac{b}{c})
	$$
	where $L \sim \Poi(padq)$ and $L' \sim \Poi((1-p)bdq)$ are independent. By isolating the means
	\begin{align} 
		\xi^{(1)}_1 \overset{(d)}{=} h &+ (L - padq) \log(\frac{a}{b}) + (L' - (1-q)bdp) \log(\frac{b}{c}) \nonumber \\
		&+ dq (pa \log(\frac{a}{b}) + (1-p)b \log(\frac{b}{c}))\label{eq:iso_mean}
	\end{align}

	In the $d \to \infty$ limit, $\log(\frac{a}{b}) = \frac{\epsilon}{p} + \frac{1}{2} \epsilon^2 (1-(\frac{1-p}{p})^2) + o(\epsilon^2)$ and $\log(\frac{b}{c}) = - \frac{\epsilon}{1-p} - \frac{1}{2} \epsilon^2 (1-(\frac{p}{1-p})^2) + o(\epsilon^2)$, so the last term in \eqref{eq:iso_mean} becomes
	\begin{align*}
		dq (pa \log\frac{a}{b} + (1-p)b \log\frac{b}{c})
		&= q \lambda ( \frac{1}{p} + \frac{2p-1}{2p(1-p)}) +o(1)
		\\
		&= \frac{q \lambda}{2 p (1-p)} + o(1)
	\end{align*}

	Now, to deal with the remaining terms in \eqref{eq:iso_mean} we are going to use the following version of the central limit theorem.

	\begin{lemma}[Central limit theorem] \label{th:clt}
		For $n \in \N^*$, let $X_1^{(n)}, \dots, X_n^{(n)} \sim_{iid} P_n$ where $P_n$ verifies $\E_{P_n} X = 0$, $E_{P_n}(X^2) = 1$ and $\sup_n \E_{P_n} \vert X \vert^3 < \infty$. We have then
		$$
		\frac{1}{\sqrt{n}} \sum_{i=1}^{n} X^{(n)}_i \xrightarrow[n \to + \infty]{W} \mathcal{N}(0,1)
		$$
	\end{lemma}
	\begin{proof}
		The convergence of the first two moments is trivial, we only have to show that the rescaled sum converges weakly toward the Gaussian.
		Let $\phi_n(t) = \E_{P_n} [ e^{itX} ]$ the Fourier transform of $P_n$. By derivation under the expectation and Taylor formula, we have
		$$
		\forall t \in \R, \vert \phi_n(t) - (1 - \frac{t^2}{2}) \vert \leq \frac{C}{3} t^3
		$$
		where $C = \sup_n \E_{P_n} \vert X \vert^3$. Let $u \in \R$
		\begin{align*}
			\E \big[ &\exp(iu \frac{1}{\sqrt{n}} \sum_{i=1}^n X_i^{(n)}) \big] \\
					 &= \phi_n(\frac{u}{\sqrt{n}})^n = (1 - \frac{u^2}{2n} + O(n^{-3/2}))^n \xrightarrow[n \to + \infty]{} e^{-u^2 / 2}
		\end{align*}
		because the ``big O'' does not depend on $n$. Levy theorem gives the result.
	\end{proof}

	\begin{corollary}
		Let $(a_n)_n \in ]0, + \infty [^{\N}$, such that $a_n \xrightarrow[n \to \infty]{} + \infty$. Let $X_n$ be a sequence of random variable such that $X_n \sim \Poi(a_n)$. Then
		$$
		\frac{1}{\sqrt{a_n}} (X_n - a_n) \xrightarrow[n \to \infty]{W} \mathcal{N}(0,1)
		$$
	\end{corollary}

	\begin{proof}
		We define $v_n := \frac{a_n}{\ceil{a_n}} \to 1$. Let $(Y_i^{(n)}) \sim_{iid} \Poi(v_n) - v_n$. We apply the previous theorem to the $Y$ and obtain
		$$
		\frac{1}{\sqrt{a_n}} (X_n - a_n) \overset{(d)}{=} \frac{1}{\sqrt{a_n}} \sum_{i=1}^{\ceil{a_n}} Y^{(n)}_i  \xrightarrow[n \to \infty]{W} \mathcal{N}(0,1)
		$$
	\end{proof}

	Applying this result to the terms in \eqref{eq:iso_mean}, we obtain $(L-padq) \log(\frac{a}{b}) \xrightarrow[d \to \infty]{W} \mathcal{N}(0, \frac{\lambda q}{p})$ and $(L'-(1-p)bdq) \log(\frac{b}{c}) \xrightarrow[d \to \infty]{W} \mathcal{N}(0, \frac{\lambda q}{1-p})$ and finally
	$$
	\xi^{(1)}_1 \xrightarrow[d \to \infty]{W} \mathcal{N}(h+ \frac{\mu_1}{2},\mu_1)
	$$

\end{proof}

The following lemma characterizes the asymptotic distribution of the messages when $d \rightarrow + \infty$

\begin{lemma} \label{lem:iter}
	Suppose that, for a fixed $r$ we have
	\begin{align*}
		\xi_r^{(1)} \xrightarrow[d \to + \infty]{W} \mathcal{N}(h + \frac{\mu_r}{2}, \mu_r) \\
		\xi_r^{(2)} \xrightarrow[d \to + \infty]{W} \mathcal{N}(h - \frac{\mu_r}{2}, \mu_r) 
	\end{align*}
	Then
	\begin{align*}
		\xi_{r+1}^{(1)} \xrightarrow[d \to + \infty]{W} \mathcal{N}(h + \frac{\mu_{r+1}}{2}, \mu_{r+1}) \\
		\xi_{r+1}^{(2)} \xrightarrow[d \to + \infty]{W} \mathcal{N}(h - \frac{\mu_{r+1}}{2}, \mu_{r+1}) 
	\end{align*}
	where
	\begin{align*}
		\mu_{r+1} &= \frac{\lambda}{(1-p)^2} \E \Big[ \frac{1}{p + (1-p)\exp(\sqrt{\mu_r}Z - \mu_r/2)} - 1 \Big] \\
							  &= G(\mu_r)
	\end{align*}
	and the expectation is taken with respect to $Z \sim \mathcal{N}(0,1)$.
\end{lemma}

\begin{proof}
	$f(\xi) = \log(\frac{\frac{a}{b} e^{\xi} + 1}{e^{\xi} + \frac{c}{b}})$ and we have $\frac{a}{b} = 1 + \frac{1}{p} \epsilon + O(\epsilon^2)$ and $\frac{c}{b} = 1 + \frac{1}{1-p} \epsilon + O(\epsilon^2)$.
	Thus
	\begin{align*}
	f(\xi) = &\log(1 + \epsilon \frac{e^{\xi}}{p(1+e^{\xi})} +O(\epsilon^2)) 
	\\
	&- \log(1 + \epsilon \frac{1}{(1-p)(1+e^{\xi})} +O(\epsilon^2))
	\end{align*}
	The Taylor-Lagrange formula ensures that
	$$
	\forall a \in [ 0,1 ] , \forall x \geq 0, \ \vert \log(1+ax) - ax + \frac{a^2 x^2}{2} \vert \leq \frac{1}{3}x^3
	$$
	Therefore
	\begin{align}
		f(x) = &\epsilon \frac{e^{x}}{p(1+e^{x})} - \frac{\epsilon^2}{2p^2} \big( \frac{e^x}{1+e^x} \big)^2 - \epsilon \frac{1}{(1-p)(1+e^{x})} \nonumber
		\\
		&+ \frac{\epsilon^2}{2(1-p)^2} \big( \frac{1}{1+e^x} \big)^2 + O(\epsilon^3) \label{eq:dl_f}
	\end{align}

	\begin{lemma}[Wald formulas]
		Let $X_1, \dots, X_n$ be iid integrables real random variables, and $T$ a $\N$-valued integrable random variable, independent of the $X_i$. Then $\sum_{i=1}^{T} X_i$ is integrable and
		$$
		\E \Big[ \sum_{i=1}^T X_i \Big] = \E [T] \E [X]
		$$
		Moreover, if the variables $(X_i)$ are square-integrable and centered, then
		$$
		Var \Big( \sum_{i=1}^T X_i \Big) = \E [T] \E (X^2)
		$$
	\end{lemma}

	Therefore
	\begin{align}
		\E \xi_{r+1}^{(1)} &= h + pad \ \E f(\xi_r^{(1)}) + (1-p)bd \ \E f(\xi_r^{(2)}) \label{eq:esp1}\\
		\E \xi_{r+1}^{(2)} &= h + pbd \ \E f(\xi_r^{(1)}) + (1-p)cd \ \E f(\xi_r^{(2)}) \label{eq:esp2}
	\end{align}

	\begin{lemma}[Nishimori condition]
		For all continuous bounded function $g$:
		$$
		\E g(\xi_r^{(2)}) = \frac{p}{1-p} \E g(\xi_r^{(1)}) e^{-\xi_r^{(1)}}
		$$
	\end{lemma}

	\begin{proof}
		This is a consequence of Bayes rule.
		\begin{align*}
			&\PP(\xi_r^{2} \in A) = \PP(\xi_r \in A | X_{s_0} = 2)
			= \frac{\PP(\xi_r \in A , X_{s_0} = 2)}{1-p} \\
			&= \frac{1}{1-p} \E(1(\xi_r \in A) \PP(X_{s_0} = 2 | GW,E_r,(X_s)_{s \in E_r})) \\
			&= \frac{1}{1-p} \E(1(\xi_r \in A) \PP(X_{s_0} = 1 | GW,E_r,(X_s)_{s \in E_r}) e^{-\xi_r}) \\
			&= \frac{p}{1-p} \E(1(\xi_r \in A)  e^{-\xi_r} | X_{s_0} = 1)
		\end{align*}
	\end{proof}

	\begin{corollary}
		\begin{align*}
			p \E \frac{e^{\xi_r^{(1)}}}{1+e^{\xi_r^{(1)}}} + (1-p) \E \frac{e^{\xi_r^{(2)}}}{1+e^{\xi_r^{(2)}}} &= p \\
			p \E \Big(\frac{e^{\xi_r^{(1)}}}{1+e^{\xi_r^{(1)}}}\Big)^2 + (1-p) \E \Big(\frac{e^{\xi_r^{(2)}}}{1+e^{\xi_r^{(2)}}}\Big)^2 &= p\E \frac{e^{\xi_r^{(1)}}}{1+e^{\xi_r^{(1)}}}
		\end{align*}
	\end{corollary}

	Replacing $f$ by its approximation \eqref{eq:dl_f} in equations \eqref{eq:esp1} and \eqref{eq:esp2} show then that $\E \xi_{r+1}^{(1)} \xrightarrow[d \to + \infty]{} h - \frac{1}{2} G(\mu_r)$ and $\E \xi_{r+1}^{(2)} \xrightarrow[d \to + \infty]{} h + \frac{1}{2} G(\mu_r)$ where
	$$
	G(\mu) \!=\! \frac{\lambda}{(1-p)^2} \E \Big[ \frac{1}{p + (1-p)\exp(\sqrt{\mu}Z - \mu/2)} - 1 \Big]
	$$
	and where $Z \sim \mathcal{N}(0,1)$.

	Similar calculations show that $Var (\xi_{r+1}^{(1)}) \xrightarrow[d \to + \infty]{} G(\mu_r)$ and $Var (\xi_{r+1}^{(2)}) \xrightarrow[d \to + \infty]{} G(\mu_r)$.

	We are now going to show that $\xi_{r+1}^{(1)}$ is converging toward a Gaussian distribution in the Wasserstein sense.

	We will need the following lemma.
	\begin{lemma}
		For $i=1,2$, $\sqrt{d} \ \E[ f(\xi_{r}^{(i)}) ]$ and $d \ \E[ f(\xi_{r}^{(i)})^2 ]$ are both converging to constants when $d \rightarrow + \infty$.
	\end{lemma}

	\begin{proof}
		$\sqrt{d} f(x) = \frac{\sqrt{d} \epsilon}{p} \frac{e^x}{1+e^x} 
		- \frac{\sqrt{d} \epsilon}{1-p} \frac{1}{1+e^x} +o(1)$. So
		$$
		\sqrt{d} \ \E[ f(\xi_{r}^{(i)}) ] 
		\!=\! \frac{\sqrt{d} \epsilon}{p} \E \Big[ \frac{e^{\xi_r^{(i)}}}{1+e^{\xi_r^{(i)}}} \Big] - \frac{\sqrt{d} \epsilon}{1-p} \E \Big[ \frac{1}{1+e^{\xi_r^{(i)}}} \Big] +o(1)
		$$
		and all the terms in this expression are converging because of weak convergence and the fact that $\epsilon \sim \sqrt{\frac{\lambda}{d}}$. The other limit is proved analogously.
	\end{proof}

	{\small
	\begin{align} 
		\xi_{r+1}^{(1)} - \E \xi_{r+1}^{(1)}& =
		\sum_{i=1}^{L_{1,1}} f(\xi_{r,i}^{(1)}) - \E f(\xi_{r,i}^{(1)})
		+ \sum_{i=1}^{L_{1,2}} f(\xi_{r,i}^{(2)}) - \E f(\xi_{r,i}^{(2)}) \nonumber
		\\
		&+
		(L_{1,1} - \E L_{1,1}) \E f(\xi_r^{(1)}) + (L_{1,2} - \E L_{1,2}) \E f(\xi_r^{(2)}) \label{eq:decomp_tcl}
	\end{align}
}

	We note $X_i = f(\xi_{r,i}^{(1)}) - \E f(\xi_{r,i}^{(1)})$ and $Y_i = f(\xi_{r,i}^{(2)}) - \E f(\xi_{r,i}^{(2)})$. Let us decompose the first sum:
	$$
	\sum_{i=1}^{L_{1,1}} X_i
	=\sum_{i=1}^{\E L_{1,1}} X_i
	+
	\underbrace{ \sum_{i=1}^{L_{1,1}} X_i - \sum_{i=1}^{\E L_{1,1}} X_i}_{S}
	$$
	We first show that $S \xrightarrow[d \to \infty]{W} \delta_0$.  Wald identities give us

	\begin{align*}
		Var(S) &= Var \big( \sum_{i=1}^{\vert L_{1,1} - \E L_{1,1} \vert } X_i \big) = \E \big[ \vert L_{1,1} - \E L_{1,1} \vert \big] \E[X_1^2] 
		\\
		&= \underbrace{ \frac{1}{d} \E \big[ \vert L_{1,1} - \E L_{1,1} \vert \big]}_{\rightarrow 0} 
		\ \underbrace{b Var \big(f(\xi_r^{(0)})\big)}_{O(1)}
	\end{align*}
	And therefore $S \xrightarrow[d \to \infty]{W} \delta_0$ because $\E S = 0$.

	Next, we apply the central limit theorem lemma~\ref{th:clt} to the sum $\sum_{i=1}^{\E L_{1,1}} X_i$
	$$
	\sum_{i=1}^{\E L_{1,1}} X_i = \frac{1}{\sqrt{(1-p) a}} \sum_{i=1}^{(1-p) a} \underbrace{\sqrt{(1-p) a} X_i
	}_{\text{iid, bounded random variables}}$$
	to obtain that the sum converges with respect to the Wasserstein metric to a normal distribution.

	The two first sums in \eqref{eq:decomp_tcl} are independent and converges to Gaussian distributions in the Wasserstein sense. It remains to show that the last two terms are converging toward Gaussian distributions. This is indeed the case because
	$$
	(L_{1,1} - \E L_{1,1}) \E f(\xi_r^{(0)}) 
	= \underbrace{ \frac{1}{\sqrt{d}} (L_{1,1} - \E L_{1,1}) }_{\rightarrow \mathcal{N}}
	\
	\underbrace{\sqrt{d} \E f(\xi_r^{(0)})}_{\rightarrow \text{ constant }}
	$$
	and the last term is treated the same way.

	$\xi_{r+1}^{(1)}$ is therefore converging toward a Gaussian distribution in the Wasserstein sense. The mean and the variance of this Gaussian distribution are necessarily equal to the limits of the means and the variance of $\xi_{r+1}^{(1)}$ that we computed.
\end{proof}
\\

Proposition~\ref{prop:result_tree} is now a consequence of Conjecture~\ref{lem:fixed_points} and the following corollary.

\begin{corollary}
	For all $0 \leq q \leq 1$,
	$$
	\lim_{n \to \infty} \lim_{d \to \infty} D_{TV} (P^{(n,q)}_1, P^{(n,q)}_2) = 2 \PP(\mathcal{N}(\mu_{\infty}/2, \mu_{\infty}) > 0) -1
	$$
	where $\mu_{\infty}$ is the limit of the sequence
	\begin{equation} \label{eq:rec_mu_q}
		\begin{cases}
			\mu_1 &=  \frac{q \lambda}{p(1-p)} \\
			\mu_{k+1} &= G(\mu_k)
		\end{cases}
	\end{equation}
\end{corollary}

\begin{proof}
	The optimal test according the performance measure $P_{succ}$ is 
\begin{equation} \label{eq:test_opt_tree}
	T^{(GW)}_r(T,E_r,(X_s)_{s \in E_r}) =
\begin{cases}
	1 & \text{if } \xi_{r} \geq \log\frac{p}{1-p} \\
	2 & \text{otherwise} \\
\end{cases}
\end{equation}

Analogously to Lemma~\ref{lem:dtv_psucc}, we have 
\begin{align*}
	D_{TV} &(P^{(n,q)}_1, P^{(n,q)}_2) = P_{succ}(T^{(GW)}_n) \\
			  &=\PP(\xi_n^{(1)} \geq \log\frac{p}{1-p}) + \PP(\xi_n^{(2)} < \log\frac{p}{1-p}) -1 \\
				 &\xrightarrow[d \to \infty]{} 2 \PP(\mathcal{N}(\mu_{n}/2, \mu_{n}) > 0) -1\\
				 &\xrightarrow[n \to \infty]{} 2 \PP(\mathcal{N}(\mu_{\infty}/2, \mu_{\infty}) > 0) -1\\
\end{align*}
where we have used lemma~\ref{lem:init} and lemma~\ref{lem:iter}.
\end{proof}

\section{Proofs for the stochastic block model}\label{sec:proofsbm}
In this section, we apply the results that we obtained for the reconstruction on the branching process to derive bounds for the community detection problem on the stochastic block model. 

Consider the case, where a fraction $0\leq q \leq 1$ is revealed: one observes the graph $G$ and additionally each label $X_v$ with probability $q$, for $1 \leq v \leq n$, independently of everything else. Let us denote $E_G = \{1 \leq v \leq n \ | \ X_v \text{ is revealed} \}$.

Let $s_0$ be uniformly chosen among the vertices of $G$. For $r \geq 0$ we define (analogously to the case of the branching process) $E_{G,r} = \partial [G]_r \cap E_G$, i.e.\ the vertices at the boundary of the ball of center $s_0$ and radius $r$, whose label as been revealed. Define

\begin{equation} \label{eq:xi_Gr}
	\xi_{G,r} = \log \left(\frac{\PP(X_{s_0} = 1 \vert [G]_r, (X_s)_{s \in E_{G,r}})}{\PP(X_{s_0} = 2 \vert [G]_r, (X_s)_{s \in E_{G,r}})} \right) 
\end{equation}

Let us denote $P_{G,r}^{(1)}$ and $P_{G,r}^{(2)}$ the laws of $\xi_{G,r}$ conditionally respectively on $X_{s_0} = 1$ and $X_{s_0}=2$.

\begin{proposition} \label{prop:convergence_xi}
	$$
	\text{For } i=1,2; \ P_{G,r}^{(i)} \xrightarrow[n \to \infty]{} P_r^{(i)}
	$$
	where $P_r^{(i)}$ is defined in Definition~\ref{def:xi_tree}.
\end{proposition}

Proposition~\ref{prop:convergence_xi} is a consequence of the local convergence of the stochastic block model towards a branching process (Theorem~\ref{th:sbm_to_gw}). 

\begin{proof}

	Let us define
	{\small
	\begin{align*}
		&F_n: (\tilde{G},\tilde{s}_0,\tilde{X}, \tilde{E}_r) \mapsto 
		\\
		&\log
			\frac
			{\PP(X_{s_0} = 1 \vert [G]_r = [\tilde{G}]_r, (X_s)_{s \in E_{G,r}} = (\tilde{X}_s)_{s \in \tilde{E}_r}, E_{G,r} = \tilde{E}_r)}
			{\PP(X_{s_0} = 2 \vert [G]_r = [\tilde{G}]_r, (X_s)_{s \in E_{G,r}} = (\tilde{X}_s)_{s \in \tilde{E}_r}, E_{G,r} = \tilde{E}_r)}
	\end{align*}
}
	and

	{\small
	\begin{align*}
		&F_{\infty}: (\tilde{G},\tilde{s}_0,\tilde{X},\tilde{E}_r) \mapsto 
		\\
		&\begin{cases}
			\!\log
				\frac
				{\PP(X_{s_0} = 1 \vert [T]_r = [\tilde{G}]_r, (X_s)_{s \in E_r} = (\tilde{X}_s)_{s \in \tilde{E}_r}, E_r = \tilde{E}_r)}
				{\PP(X_{s_0} = 2 \vert [T]_r = [\tilde{G}]_r, (X_s)_{s \in E_r} = (\tilde{X}_s)_{s \in \tilde{E}_r}, E_r = \tilde{E}_r)}
			& \!\!\!\! \text{if } [\tilde{G}]_r \text{ is a tree} \\
			\!0 &\!\!\!\! \text{else} \\
		\end{cases}
	\end{align*}
}

	
	Let $(\tilde{G},\tilde{s}_0,\tilde{X})$ be a fixed pointed labeled graph such that $[\tilde{G}]_r$ is a tree. Let $\tilde{E}_r$ a subset of the vertices in $\partial [\tilde{G}]_r$.
	A straightforward extension of theorem~\ref{th:sbm_to_gw} gives us
	\begin{equation} \label{eq:cv_fn}
		F_n(\tilde{G},\tilde{s}_0,\tilde{X},\tilde{E}_r) \xrightarrow[n \to \infty]{} F_{\infty}(\tilde{G},\tilde{s}_0,\tilde{X},\tilde{E}_r)
	\end{equation}

	Another consequence of the local convergence of $(SBM_n)$ toward $GW$ is that one can couple $(SBM_n,E_{G,r})_n$ and $(GW,E_r)$ on a probability space such that there exists $n_0 \in \N$ such that
	$$
	\forall n \geq n_0, [SBM_n,E_{G,r}]_r = [GW,E_r]_r
	$$
	Let $n \geq n_0$
	$$
	F_n(SBM_n,E_{G,r}) = F_n(GW,E_r) \xrightarrow[n \to \infty]{} F_{\infty}(GW,E_r)
	$$
	On this probability space $F_n(SBM_n,E_{G,r})$ converges almost surely to $F_{\infty}(GW,E_r)$, hence the convergence of the conditional distributions.
\end{proof}
\\

Define the local test
\begin{equation} \label{eq:def_tloc}
T^{loc}_r(G,X) =
\begin{cases}
	1 & \text{if } \xi_{G,r} \geq \log\frac{p}{1-p} \\
	2 & \text{otherwise} \\
\end{cases}
\end{equation}
Using the results on the branching process, we are now able to fully characterize the performance of $T^{loc}$.

\begin{proposition} \label{prop:local_test}
	For all $0 \leq q \leq 1$,
	$$
	\lim_{r \to \infty}	\lim_{d \to \infty} \lim_{n \to \infty}  P_{succ}(T^{loc}_r) = 2 \PP(\mathcal{N}(\mu_{\infty}/2, \mu_{\infty})>0)-1
	$$
	where $\mu_{\infty}$ is the limit of the sequence defined by \eqref{eq:rec_mu_q}.
\end{proposition}

\begin{proof}
Using Proposition~\ref{prop:convergence_xi} and Lemma~\ref{lem:iter}:
$$
\begin{cases}
	P_{G,r}^{(1)}
	\xrightarrow[n \to \infty]{(d)}
	P_r^{(1)}
	\xrightarrow[d \to \infty]{(d)}
	\mathcal{N}(h + \frac{\mu_r}{2},\mu_r)
	\\
	P_{G,r}^{(2)}
	\xrightarrow[n \to \infty]{(d)}
	P_r^{(2)}
	\xrightarrow[d \to \infty]{(d)}
	\mathcal{N}(h - \frac{\mu_r}{2},\mu_r)
\end{cases}
$$

Where $\mu_r$ is defined by the recursion \eqref{eq:rec_mu_q}, by Lemmas~\ref{lem:init} and~\ref{lem:iter}. Therefore (recall that $h= \log\frac{p}{1-p}$)
\begin{align*}
	P_{succ}(T^{loc}_r)& = \PP(\xi_{G,r} \geq h | X_{s_0} = 1) + \PP(\xi_{G,r} < h | X_{s_0} = 2) - 1 \\
					   &\xrightarrow[n \to \infty]{}
	P_{G,r}^{(1)}(\xi_{G,r} \geq h) + P_{G,r}^{(2)}(\xi_{G,r} < h) - 1 \\
	&\xrightarrow[d \to \infty]{}
	2 \PP(\mathcal{N}(\mu_r/2,\mu_r) > 0) - 1 \\
	&\xrightarrow[r \to \infty]{}
	2 \PP(\mathcal{N}(\mu_{\infty}/2,\mu_{\infty}) > 0) - 1 
\end{align*}

\end{proof}

\subsection{Upper bound}

We can now deduce an upper bound on the optimal performance for the community detection problem.
\begin{corollary} \label{cor:upper_bound}
  We have:
  $$
	\limsup_{d \to \infty} \limsup_{n \to \infty} P_{succ}(T^{opt}) \leq 2 \PP(\mathcal{N}(\mu_{\infty}/2, \mu_{\infty})>0)-1
	$$
	where $\mu_{\infty}$ is the limit of the sequence defined by \eqref{eq:rec_mu_q}.
\end{corollary}

Proposition~\ref{prop:below_sp} and the first part of Proposition~\ref{prop:ks} are then consequences of this corollary:
\begin{itemize}
	\item when $\lambda < \lambda_{sp}(p)$, $\mu_{\infty}=0$ and consequently $\lim_{d \to \infty} \limsup_{n \to \infty} P_{succ}(T^{opt})=0$.
	\item when $\lambda > 1$, then $\mu_{\infty} = \alpha$, hence the first bound of Proposition~\ref{prop:ks}.
\end{itemize}


\begin{proof}
Let $r > 0$. Let $s_0$ be uniformly chosen from the vertices of $G$. We aim at estimating $X_{s_0}$ from the rooted graph $(G,s_0)$. As seen in section~\ref{section:community_detection}, the optimal test in term of rescaled success probability $P_{succ}$ is 
$$
T^{opt}(G) =
\begin{cases}
	1 & \text{if } \log(\frac{\PP(X_{s_0} = 1 \vert G)}{\PP(X_{s_0} = 2 \vert G)}) \geq \log(\frac{p}{1-p}) \\
	2 & \text{otherwise} \\
\end{cases}
$$

We are going to analyze the oracle
$$
T^*_r(G,X) =
\begin{cases}
	1 & \text{if } \log\frac{\PP(X_{s_0} = 1 \vert G, (X_s)_{s \in \partial [G]_r})}{\PP(X_{s_0} = 2 \vert G, (X_s)_{s \in \partial [G]_r})} \geq \log\frac{p}{1-p} \\
	2 & \text{otherwise} \\
\end{cases}
$$

Obviously, $P_{succ}(T^{opt}) \leq P_{succ}(T^*_r)$. The oracle $T^*_r$ uses extra information $(X_s)_{s \in \partial [G]_r}$ but is a local test, i.e.\ a test that only depends on the ball of radius $r$:
\begin{align*}
	\xi_{G,r}^* &:= \log \Big(\frac{\PP(X_{s_0} = 1 \vert G, (X_s)_{s \in \partial [G]_r})}{\PP(X_{s_0} = 2 \vert G, (X_s)_{s \in \partial [G]_r})} \Big) \\
			  &= \log \Big(\frac{\PP(X_{s_0} = 1 \vert [G]_r, (X_s)_{s \in \partial [G]_r})}{\PP(X_{s_0} = 2 \vert [G]_r, (X_s)_{s \in \partial [G]_r})} \Big) 
\end{align*}

$\xi^*_{G,r}$ is thus equal to $\xi_{G,r}$ from equation \eqref{eq:xi_Gr}, when $q=1$. We conclude using proposition~\ref{prop:local_test} and the fact $P_{succ}(T^{opt}) \leq P_{succ}(T_r^*)$.
\end{proof}

\subsection{Lower bound}

We now establish a lower bound for estimation when a fraction $q$ of the labels is revealed.

\begin{corollary} \label{cor:lower_bound}
	For all $0 \leq q \leq 1$,
	\begin{equation} \label{eq:lower_bound_local}
		\liminf_{d \to \infty} \liminf_{n \to \infty} P_{succ}(T^{opt}(G,q)) \geq 2 \PP(\mathcal{N}(\mu_{\infty}/2, \mu_{\infty})>0)-1
	\end{equation}

	where $\mu_{\infty}$ is the limit of the sequence defined by \eqref{eq:rec_mu_q}.
\end{corollary}

The second part of Proposition~\ref{prop:ks} follows from this corollary. Indeed, when $\lambda > 1$ and $q>0$, $\mu_{\infty} = \alpha >0$.

Corollary~\ref{cor:lower_bound} leads also to proposition~\ref{prop:belowks}: when $q >  \frac{\beta p (1-p)}{\lambda}$, then $\mu_1 > \beta$ and thus $(\mu_k)$ converges to the fixed point $\alpha > \beta$ of $G$: $\mu_{\infty} = \alpha > 0$.

We deduce also Proposition~\ref{prop:lower_bound_ks} from the proof of Corollary~\ref{cor:lower_bound}. Indeed, we will see that the lower bound in \eqref{eq:lower_bound_local} is achieved by a local test.

\begin{proof}
	Here, we are going to bound by below $P_{succ}(T^{opt}(G,q))$ by the performance of the local test $T^{loc}_r(G,X)$ defined by \eqref{eq:def_tloc}.
Obviously,
\begin{align*}
	P_{succ}(T^{opt}(G,q)) \geq P_{succ}(T^{loc}_r) 
\end{align*}
Proposition~\ref{prop:local_test} gives then the result.
\end{proof}

\bibliographystyle{IEEEtran}
\bibliography{IEEEabrv,references}

\begin{thebibliography}{10}
\providecommand{\url}[1]{#1}
\csname url@samestyle\endcsname
\providecommand{\newblock}{\relax}
\providecommand{\bibinfo}[2]{#2}
\providecommand{\BIBentrySTDinterwordspacing}{\spaceskip=0pt\relax}
\providecommand{\BIBentryALTinterwordstretchfactor}{4}
\providecommand{\BIBentryALTinterwordspacing}{\spaceskip=\fontdimen2\font plus
\BIBentryALTinterwordstretchfactor\fontdimen3\font minus
  \fontdimen4\font\relax}
\providecommand{\BIBforeignlanguage}[2]{{%
\expandafter\ifx\csname l@#1\endcsname\relax
\typeout{** WARNING: IEEEtran.bst: No hyphenation pattern has been}%
\typeout{** loaded for the language `#1'. Using the pattern for}%
\typeout{** the default language instead.}%
\else
\language=\csname l@#1\endcsname
\fi
#2}}
\providecommand{\BIBdecl}{\relax}
\BIBdecl

\bibitem{decelle2011asymptotic}
A.~Decelle, F.~Krzakala, C.~Moore, and L.~Zdeborov{\'a}, ``Asymptotic analysis
  of the stochastic block model for modular networks and its algorithmic
  applications,'' \emph{Physical Review E}, vol.~84, no.~6, p. 066106, 2011.

\bibitem{mossel2015symSBM1}
E.~Mossel, J.~Neeman, and A.~Sly, ``Reconstruction and estimation in the
  planted partition model,'' \emph{Probability Theory and Related Fields}, vol.
  162, no. 3-4, pp. 431--461, 2015.

\bibitem{massoulie2014symSBM}
L.~Massouli{\'e}, ``Community detection thresholds and the weak ramanujan
  property,'' in \emph{Proceedings of the 46th Annual ACM Symposium on Theory
  of Computing}.\hskip 1em plus 0.5em minus 0.4em\relax ACM, 2014, pp.
  694--703.

\bibitem{mossel2013symSBM2}
E.~Mossel, J.~Neeman, and A.~Sly, ``A proof of the block model threshold
  conjecture,'' \emph{arXiv preprint arXiv:1311.4115}, 2013.

\bibitem{neeman2014asymSBM}
J.~Neeman and P.~Netrapalli, ``Non-reconstructability in the stochastic block
  model,'' \emph{arXiv preprint arXiv:1404.6304}, 2014.

\bibitem{bordenave2015non}
C.~Bordenave, M.~Lelarge, and L.~Massouli{\'e}, ``Non-backtracking spectrum of
  random graphs: community detection and non-regular ramanujan graphs,'' in
  \emph{Foundations of Computer Science (FOCS), 2015 IEEE 56th Annual Symposium
  on}.\hskip 1em plus 0.5em minus 0.4em\relax IEEE, 2015, pp. 1347--1357.

\bibitem{mossel2004survey}
E.~Mossel, ``Survey-information flow on trees,'' \emph{DIMACS series in
  discrete mathematics and theoretical computer science}, vol.~63, pp.
  155--170, 2004.

\bibitem{mossel2015density}
E.~Mossel and J.~Xu, ``Density evolution in the degree-correlated stochastic
  block model,'' \emph{arXiv preprint arXiv:1509.03281}, vol.~7, 2015.

\bibitem{kanade2014global}
V.~Kanade, E.~Mossel, and T.~Schramm, ``Global and local information in
  clustering labeled block models,'' 2014.

\bibitem{heimlicher2012community}
S.~Heimlicher, M.~Lelarge, and L.~Massouli{\'e}, ``Community detection in the
  labelled stochastic block model,'' \emph{arXiv preprint arXiv:1209.2910},
  2012.

\bibitem{lelarge2015reconstruction}
M.~Lelarge, L.~Massouli{\'e}, and J.~Xu, ``Reconstruction in the labelled
  stochastic block model,'' \emph{IEEE Transactions on Network Science and
  Engineering}, vol.~2, no.~4, pp. 152--163, 2015.

\bibitem{saade2016clustering}
A.~Saade, M.~Lelarge, F.~Krzakala, and L.~Zdeborov{\'a}, ``Clustering from
  sparse pairwise measurements,'' \emph{arXiv preprint arXiv:1601.06683}, 2016.

\bibitem{saade2016fast}
A.~Saade, F.~Krzakala, M.~Lelarge, and L.~Zdeborov{\'a}, ``Fast randomized
  semi-supervised clustering,'' \emph{arXiv preprint arXiv:1605.06422}, 2016.

\bibitem{krzakala2016lowrank}
F.~Krzakala, J.~Xu, and L.~Zdeborov{\'a}, ``Mutual information in rank-one
  matrix estimation,'' \emph{arXiv preprint arXiv:1603.08447}, 2016.

\bibitem{barbier2016mutual}
J.~Barbier, M.~Dia, N.~Macris, F.~Krzakala, T.~Lesieur, and L.~Zdeborova,
  ``Mutual information for symmetric rank-one matrix estimation: A proof of the
  replica formula,'' \emph{arXiv preprint arXiv:1606.04142}, 2016.

\bibitem{evans2000broadcasting}
W.~Evans, C.~Kenyon, Y.~Peres, and L.~J. Schulman, ``Broadcasting on trees and
  the ising model,'' \emph{Annals of Applied Probability}, pp. 410--433, 2000.

\bibitem{mossel2003information}
E.~Mossel and Y.~Peres, ``Information flow on trees,'' \emph{The Annals of
  Applied Probability}, vol.~13, no.~3, pp. 817--844, 2003.

\bibitem{montanari2015finding}
A.~Montanari, ``Finding one community in a sparse graph,'' \emph{Journal of
  Statistical Physics}, vol. 161, no.~2, pp. 273--299, 2015.

\end{thebibliography}

\end{document}